\documentclass[EJP]{ejpecp} 

\usepackage{graphicx}  

\SHORTTITLE{Asymptotic behavior of the Brownian frog model} 
   
\TITLE{Asymptotic behavior of the Brownian frog model} 

\AUTHORS{%
  Erin~Beckman\footnote{Duke University.
    \EMAIL{ebeckman@math.duke.edu}}
  \and 
  Emily~Dinan  \footnote{University of Washington \EMAIL{	emdinan1@gmail.com}}
  \and 
  Rick~Durrett \footnote{Duke University. \EMAIL{rtd@math.duke.edu}}
  \and
  Ran Huo \footnote{Duke University.\EMAIL{ran.huo@duke.edu}}
  \and
  Matthew Junge \footnote{Duke Unversity. \EMAIL{jungem@math.duke.edu}}
  }


\AMSSUBJ{60J25} 

\SUBMITTED{December 15, 2017} 
\ACCEPTED{August 21, 2018} 

\VOLUME{0}
\YEAR{2016}
\PAPERNUM{0}
\DOI{10.1214/YY-TN}

\ABSTRACT{We introduce an extension of the frog model to Euclidean space and prove properties for the spread of active particles. Fix $r>0$ and place a particle at each point $x$ of a unit intensity Poisson point process $\mathcal P \subseteq \mathbb R^d - \B(0,r)$. Around each point in $\mathcal{P}$, put a ball of radius $r$.  
A particle at the origin performs Brownian motion. When it hits the ball around $x$ for some  $x \in \mathcal P$, new particles begin independent Brownian motions from the centers of the balls in the cluster containing $x$. Subsequent visits to the cluster do nothing. This waking process continues indefinitely.  For $r$ smaller than the critical threshold of continuum percolation, we show that the set of activated points in $\mathcal P$ approximates a linearly expanding ball. Moreover, in any fixed ball the set of active particles converges to a unit intensity Poisson point process.}

\usepackage{amssymb}
\usepackage{tikz}
\usepackage{pgf}
\usetikzlibrary{decorations.pathmorphing,patterns} 
\usepackage{pgfplots}
\usepgfplotslibrary{fillbetween}
\usepackage{theoremref}
\usepackage{amsmath,dsfont,amsthm}
\usepackage{enumitem}
\usepackage{caption,subcaption}

\newcommand{\E}{\mathbf{E}}
\newcommand{\expE}{\mathbf{E}}
\renewcommand{\P}{\mathbf{P}}
\newcommand{\Pm}{\mathbf{P}}

\newcommand{\Rm}{\mathbb{R}}
\newcommand{\B}{\mathbb{B}}

\newcommand{\f}{\frac}
\newcommand{\no}{\nonumber}
\newcommand{\ind}[1]{\mathbf 1 \{ #1 \}}

\newcommand{\e}{\text{e}}

\newcommand{\ep}{\epsilon}

\usetikzlibrary{calc}

\begin{document}



\section{Introduction}

The \emph{Brownian frog model} starts with a single active particle  at the origin and sleeping particles at the points of a unit intensity Poisson point process $\mathcal{P}\subset \mathbb R^d - \B(0,r)$ with $r > 0$ fixed. Let the cluster associated with $x \in \mathcal{P}$ be the largest component of $\cup_{y \in \mathcal{P}} B(y,r)$ containing $B(x,r)$. Active particles perform independent Brownian motion, and sleeping particles become active the first time an active particle is within $r$ of its cluster. All particles in the cluster wake simultaneously. 

One way to visualize the process is that each particle is a frog on a lily pad of radius $r$. If a frog wakes up, it strays from its starting location according to a Brownian motion and wakes sleeping frogs upon perturbing their cluster of lily pads. 
 To avoid an explosion of active particles we must avoid the possibility of an infinite component of overlapping lily pads. This amounts to taking $r < r_d$ with $r_d$ the critical threshold in continuum percolation on $\mathbb R^d$ (see \cite{meester}).  In this article we prove that the set of activated sites at time $t$ behaves like a ball of radius $\gamma_rt$.  Additionally, we show that the limiting density of awake frogs in a fixed region is Poisson distributed.

This is the continuous analogue of a growing system of random walks known as the \emph{frog model}. The process starts with an active particle at the root of a graph and some number of sleeping particles at all the other vertices. In either continuous or discrete time,  the active particle at the root begins a random walk. Whenever an active particle visits a vertex with sleeping particles all of those particles become active and begin their own independent random walk. 

The frog model was originally proposed as a modification to an information spreading model introduced by K.\ Ravishankar. In its earliest form, the frog model was stated on $\mathbb Z^d$ with one particle-per-site and simple random walk paths. Initially it was known as the egg model, but was shortly thereafter relabeled the frog model. The ``branching," i.e.\ activation of sleeping particles, is constrained by the number of points on each site and the paths of all frogs in the process. This makes for a rather nuanced process that exhibits behavior different from a single or a branching random walk. Likewise, the Brownian frog model has features somewhere between Brownian motion and branching Brownian motion.

Early papers on the frog model consider the process with one particle-per-site on $\mathbb{Z}^d$. The paper \cite{telcs1999} shows that the root is visited infinitely often in all dimensions. Next, \cite{shape, combustion} proved that the set of visited sites has a limiting shape. The latter works in continuous time, so the proofs and theorem statements are slightly different. 
The past two decades have further investigated the process by modifying the graph, the number of frogs per site, and the random paths followed by frogs. For example, \cite{random_shape, random_recurrence} consider the change in transience/recurrence and the limiting shape with a random configuration of frogs. The articles \cite{recurrence, new_drift} consider the frog model on the lattice where particles perform a biased random walk. Another modification to the random walk path is where frogs take a geometric number of steps then perish. Various phase transitions are explored in \cite{phasetree, po2}. The question of transience and recurrence on trees is explored in \cite{HJJ1, HJJ2, JJ3_log, 23tree}. Recent work has looked at the waking behavior on finite trees \cite{hermon}, and the passage time to a distinguished set of vertices in $\mathbb Z^d$ \cite{passage}. The article \cite{HJJ4} establishes linear expansion of the set of activated frogs on regular trees, and \cite{HJJ4_cover} pins down different regimes for rapid and slow cover time for the frog model on finite trees.

 A Brownian frog model on $\mathbb R$ is studied in \cite{josh2}. Active particles have a fixed leftward drift, and are placed according to a Poisson point process with intensity $f(x)$. This paper establishes sharp conditions on $f$ that determine whether the model is transient or not. Note that on the line particles will collide so it is natural to set $r=0$. In higher dimensions it is necessary to set $r>0$.

Symmetry ensures that if the Brownian frog model has a limiting shape it will be a ball, which is not the case in the one-per-site discrete frog model on $\mathbb Z^d$ for large enough $d$ (see \cite[Theorem 1.1]{combustion}). Moreover, \cite[Theorem 1.2]{shape} proves that if enough frogs are placed at each site of $\mathbb Z^d$ then the limiting shape contains a flat edge.  Here we confirm the conjecture from the introduction of \cite{combustion} that there is a limiting shape for the Brownian frog model.  We also set the stage for studying the continuous frog model. The hallmark theorems for the frog model ought to have analogues in the Brownian frog model. However, the different geometry introduces several new  lines of inquiry. It is natural to consider diffusions other than Brownian motion. For instance, an interesting question would be to consider frogs that move according to an Ornstein-Uhlenbeck process $dB_t - \alpha|X_t|\,dt$, translated to the location where they wake up. One can then ask whether the shape theorem holds for this more general continuous frog model. 

The geometry of the overlapping lily pads is another interesting feature of the Brownian frog model. As $r$ nears the threshold in continuum percolation there will be large clusters of overlapping lily pads. These can cause big jumps that take no time. Like bond percolation, continuum percolation is best understood when $d=2$. It is known that clusters are almost surely finite in this dimension \cite{meester}. Nonetheless, they have infinite expected size. How the frog model expands at criticality remains open. It may be a very hard question. Like in bond percolation on $\mathbb Z^d$, it is a major open problem to show that clusters are almost surely finite at criticality in higher dimensions. A more tractable further question is to show that the speed of expansion $\displaystyle \lim_{r \to r_d^{-}}\gamma_r = \infty$ for all $d\geq 2$. 



\subsection*{Formal description and theorem statements}
Fix $r >0$. Let $\mathcal P \subseteq \mathbb R^d - \B(0,r)$ be a unit intensity Poisson point process. We let $\{ (B_t^x)_{t \geq 0} \colon x \in \mathcal P \cup \{0\} \}$ be a family of independent Brownian motions, with $B_0^x = x$. For $x,z \in \mathcal P \cup \{0\}$ define $t(x,z) = \inf \{ t \colon \| B_t^x - z \| \leq r \}$ to be the first time the path started at $x$ is within $r$ of $z$.  Following \cite{shape}, we define the passage time to $z$ for the process started at $x$ as
	\begin{align}
		T(x,z) = \inf \Bigl \{ \sum_{ \{x_1,\hdots, x_k\} \subseteq \mathcal P} t(x_i,x_{i+1}) \colon  x_1 = x \text{ and } x_k = z \text{ for some $k$} \Bigr\}. \label{eq:T}
	\end{align}
We define $Z_z^x(t)$ to be the location at time $t$ of the particle initially at $z$ in the process started from $x$. This can be written in terms of $T(x,z)$:
\begin{align}
Z_z^x(t) = \begin{cases}
 		z, & T(x,z) \geq t \\
 		B_{t - T(x,z)}^z, & T(x,z) < t
 \end{cases}	\label{eq:Z}.
\end{align}
Now let $\Gamma(x_0, \mathcal P)\subseteq \mathbb R^d$ be the \emph{continuum percolation cluster of $\mathcal P$ containing $x_0$}. This is obtained recursively by setting $\Gamma(x_0,\mathcal P) = \bigcup_{i=1}^\infty \Gamma_i(x_0,\mathcal P)$ where
$\Gamma_0(x_0,\mathcal P) = \{x \in \mathcal P \colon \|x-x_0\| \leq r\}$, and for $i \geq 1$ 
	\begin{align*}
				\Gamma_{i}(x_0,\mathcal P) = \{ x \in \mathcal P\colon \|x - \Gamma_{i-1}(x_0,\mathcal P)\| \leq r\} .		
	\end{align*}
Define, for $x \in \mathcal{P}$, 
$$\xi_t^x = \{ z \in \mathcal P \cup \{0\} \colon T(x,y) \leq t \text{ for some } y \in \Gamma(z,\mathcal{P}) \}$$
to be the set of activated sites from $\mathcal P \cup \{0\}$ provided the initially active particle is at $x$. For convenience we set $\xi_t^0 = \xi_t.$ Define $A_t = \{ Z_z(t) \colon 	z \in \xi_t \}$ to be the locations of active frogs at time $t$.

\begin{theorem} \thlabel{thm:main}
For all $d \geq 1$, given $0<r< r_d$ there exists $0<\gamma_r<\infty$ such that for all $\epsilon >0$
$$\mathbb B(0,\gamma_r(1-\epsilon)t) \cap \mathcal P \subseteq \xi_t \subseteq \mathbb B(0,\gamma_r(1+\epsilon)t)$$
almost surely for sufficiently large $t$.
\end{theorem}
\medskip
We will give an overview of this proof momentarily. First, we state our other theorem. 
This is the continuous analogue of 
\cite[Theorem 1.3]{combustion}. Our result is stronger though, because we allow the region that behaves like a Poisson point process to grow linearly.

\begin{theorem} 
\thlabel{thm:ppp} 

For any $\ep>0$ the locations of awake frogs in $\B(0,\gamma_r(1-\ep)t)$ converge to a Poisson process. To be precise, if we let $v \in \B(0,\gamma_r(1- \epsilon))$
then the distribution of frogs in the system shifted by $vt$ converge to a Poisson process.
\end{theorem}

\subsection*{Outline of \thref{thm:main}}

As in \cite{shape,combustion} this result relies on the sub-additive ergodic theorem. Although the outline is similar to these previous works, the techniques we use are different. The proof of \thref{thm:main} relies on three propositions. We list them along with a sketch of each proof. The most important is the existence of a limiting speed in each direction.

\begin{proposition} \thlabel{prop:ergodic}
Fix a unit vector $v$ and let $L_{ v} = \{t  v \colon t \geq 0\}$ be the ray through $ v$ emanating from $x_0=0$. Let $x_n$ be the point in $\mathcal P$ nearest to $n v$ for $n \geq 0$. We have $$\lim_{n \to \infty} \f {T(x_0,x_n)}{n} = \tilde \gamma_r \in [0,\infty).$$ 
Note that the time-constant $\tilde \gamma_r$ does not depend on $ v$. 
\end{proposition}
This is a standard application of Liggett's subadditive ergodic theorem, shown below as \thref{thm:liggett}. The main difficulty is showing that $\E T(0,x_1)$ is finite, but this follows from \thref{lem:step1} which is the first step of proving \thref{prop:lower}. 

In \cite{shape} the frog model can grow at most linearly, and in the continuous time setting of \cite{combustion} it is a minor nuisance, but not a serious obstacle to show the process stays contained in a linearly expanding ball. A concern for the Brownian frog model is that overlapping clusters of frogs all wake simultaneously. So, we might have large jumps that take no time. This introduces some extra work to prove containment in a linearly expanding ball.
 
\begin{proposition}\thlabel{prop:upper}

For $0 < r < r_d$ there exists $R_1<\infty$ and $c_1>0$ such that  $$\P[\xi_t \subseteq  \B(0,R_1t)] \geq 1- e^{-c_1t} $$ for all $t\geq 0$.	
\end{proposition}

We dominate the frog model by a branching process that resets the explored regions at each branching event. We rely on large deviation estimates for the cluster size in continuum percolation and the displacement of Brownian motion. It takes some care to create a branching process that grows slowly, and couples to the Brownian frog model.

A standard ingredient in proving a shape theorem is that the process grows at least linearly. 
\begin{proposition}\thlabel{prop:lower}

Let $\overline \xi_t$ be the convex hull of the activated points in $\mathcal{P}$ at time $t$. For $0 < r < r_d$, all $\epsilon >0$, and all $m \geq 1$ there exists an almost surely finite $t_m$ such that  $$\P[\B(0,(\gamma_r - \epsilon) t) \subseteq \overline \xi_t] \geq 1 - t^{-m}$$ for all $t \geq t_m$.	
\end{proposition}
Our proof is somewhat simpler than what was done in
 \cite{shape} and \cite{combustion}. 
We write wvhp as short for ``with very high probability," which means that the probability of failure goes to 0 faster than $n^{-m}$ for any $m$. We break $[-n,n]^d$ into small cubes of size $n^{1-\beta}$ and show that by time $n^{2+\ep}$ wvhp a frog has been awoken in each small cube. Given $a>0$ we can pick $D$ independent of $n$ so that by iterating this construction $D$ times, we have an awoke frog in each of a collection of cubes with side $n^a$. These frogs are enough to finish waking the rest of the frogs in $[n,n]^d$ by time $(D+1)n^{2+\epsilon}$. We then use this lemma and the approximate independence of distant frogs to show that this is enough to hit every point in the ball expanding at rate $\gamma_r - \epsilon$.

\subsection*{Organization of paper}
Section \ref{sec:shape} establishes the existence of a time-constant $\tilde \gamma_r$ in \thref{prop:ergodic} and uses this along with Proposition \ref{prop:upper} and \ref{prop:lower} to deduce \thref{thm:main}. 
Section \ref{sec:finite} contains the proof of \thref{prop:upper}, that the Brownian frog model is contained in a linearly expanding ball. In Section \ref{sec:linear} we prove that the Brownian frog model contains a linearly expanding ball in \thref{prop:lower}. Section \ref{sec:ppp} proves that the locations of active frogs converge to a Poisson process.

\section{Existence of a limiting shape}
\label{sec:shape}

We deduce a limiting speed by showing that the passage times between particles on a fixed ray from the origin satisfy the hypotheses of Liggett's subadditive theorem. 
Deducing a limiting shape for the Brownian frog model is slightly more convenient than the discrete frog model, since it is already defined on $\mathbb R^d$. This removes the need to interpolate from $\mathbb Z^d$ to $\mathbb Q^d$ to $\mathbb R^d$. Additionally, radial symmetry ensures that the limiting speed in each direction is the same, and so the final shape must be a ball. 
 We begin by stating Liggett's subadditive ergodic theorem, then showing the passage times along a ray satisfy the hypotheses. 

\begin{theorem}[Liggett's subadditive ergodic theorem] \thlabel{thm:liggett} Suppose that $\{Y(m,n)\}$ is a collection of positive random variables indexed by integers satisfying $0 \leq m <n$ such that 
	\begin{enumerate}[label = (\roman*)]

    \item $Y(0,n) \leq Y(0,m) + Y(m,n)$ for all $0 \leq m < n$ (subadditivity);
    \item The joint distribution of $\{ Y(m+1, m+k +1), k \geq 1 \}$ is the same as that of $\{ Y(m, m+k), k \geq 1 \}$ for each $m \geq 0;$
    \item For each $k \geq 1$ the sequence of random variables $\{ Y(nk, (n+1)k), n \geq 1 \}$ is a stationary ergodic process;
    \item $\E Y(0,1) < \infty.$
\end{enumerate}
Then
$$\lim_{n \to \infty} \f{ Y(0,n) }{ n} \to \gamma \qquad a.s.,$$
where $\gamma = \inf_{ n \geq 0 } \f{ \E Y(0,n) } n.$

\end{theorem}

\begin{proof}[Proof of \thref{prop:ergodic}]

It suffices to prove that the collection $\{T(x_m,x_n)\}_{0 \leq m <n}$ satisfies (i)-(iv) of \thref{thm:liggett}. 
%
For (i), we claim that 
$$T(x_m,x_n) \leq T(x_m,x_{\ell}) + T(x_{\ell},x_n) \text{ for all $0 \leq m \leq \ell <n $}.$$ If site $x_n$ is reached before site $x_\ell$ (so that $T(x_m,x_n) \leq T(x_m,x_\ell))$, there is nothing to prove. If that does not happen and $T(x_m,x_n) > T(x_m,x_\ell)$ then the process which began with only the frog at site $x_\ell$ awake can be coupled to the process originating from site $x_m$, which might have other particles awake at time $T(x_m,x_\ell)$. Thus the remaining time to reach site $x_n$ for the original process can be at most $T(x_\ell,x_n)$, proving (i).
The second point, (ii), follows from radial symmetry and translation invariance of the initial configuration. 

 To get (iii), for simplicity assume that $L_v$ is the nonnegative $x$-axis in two dimensions. Let $\Pi_m$ be the Poisson points in the region $(m - 1/2, m + 1/2) \times \mathbb R$.
 This is an i.i.d.\ sequence and hence ergodic. Moreover, $T(x_{nk}, x_{(n+1)k})$ is a shift invariant function of the $\Pi_m$ and hence is ergodic by \cite[Theorem 7.1.3]{durrett}. 
 
 Finally, to see (iv), for $d \geq 2$ we apply \thref{lem:step1}. This guarantees the the frog process contains a ball of radius $\Omega(t^{1/2-\epsilon})$ for small $\epsilon$ with very high probability. Here $\Omega(t^{1/2-\epsilon})$ means that $f(t) = \Omega(g(t))$ if and only if $\lim_{t\to\infty} \left|{g(t)}/{f(t)}\right| = 0$. Thus, the expected time to reach $x_1$ must be finite. For $d=1$, let $\tau$ be the first return to the origin of a Brownian motion. Since $\P[\tau > t ] \leq C t^{-1/2}$ and in a random time with exponential tail we will have at least three awakened particles moving independently, we can conclude that $\E T(0,x_1) < \infty.$  
\end{proof}

Now we explain how the existence of $\gamma_r$ gives a limiting shape.

\begin{proof}[Proof of \thref{thm:main}]
The existence of $0<	\gamma_r\leq	\infty$ follows from \thref{prop:ergodic}. It is finite by \thref{prop:upper}, and positive by its definition: $\tilde \gamma_r = \inf_{n \geq 1} \E T(0, x_n)/n. < \infty$. When $d=1$ this immediately yields the shape theorem. For $d\geq 2$ a standard shape theorem argument (see \cite{shape_example, durrett_shape, shape, combustion}) where we cover $\B(0,\gamma_r)$  with finitely many balls of radius $\epsilon' R_1>0$ and scale by $t$ while applying \thref{prop:lower} gives that $\overline \xi_t \subseteq \B(0, (\gamma_r + \epsilon) t)$ for all large $t$. 
The fact that $\B(0, (\gamma_r - \epsilon) t) \subseteq \overline \xi_t$ is proven in \thref{prop:lower}.  
\end{proof}

\section{Containment in a ball} \label{sec:finite}

We will show that the set of activated sites does not grow super-linearly. The idea is to dominate it by a branching process that places particles according to the geometry of a cluster in (subcritical) continuum percolation.
We begin with an informal description. 

The initial particle at $0$ starts  inside an empty ball of radius $r<r_d$ in the environment $\mathcal P$. It performs Brownian motion until it comes within $r$ of a point $x_0 \in \mathcal P$ at time $\tau$. 
This particle has explored the Brownian sausage around $B_t$ for times in $[0,\tau)$ and found no points from $\mathcal P$ in it. To preserve a coupling with the frog model we need to, whenever possible, use the points in $\mathcal P$. Accordingly, we refill this explored region with a fresh Poisson point process, and then sample the continuum cluster at $x_0$. If we skipped this replenishment of points, then future clusters we sample would have some dependence on $\tau$ and not be identically distributed. 
After replenishing, the continuum cluster at $x_0$ will contain $K$ additional centers of balls: $x_1,\hdots, x_K$. The process then branches into $K+2$ offspring. The `$+2$' comes from the initial particle at $0$ which is placed at $x_{-1} = B_\tau$,  and the particle at $x_0$. 

For each $x_i$ we want to reproduce the environment seen by the particle initially at 0. To do this we need to refill the parts of the continuum cluster outside of $\B(x_i,r)$ with a fresh Poisson point process. We use the same points for all of the $x_i$, but fill in slightly different regions with them. 
Now, we let the particle at each $x_i$ explore the modified $\mathcal P$ environment 
until it starts a new branching event which occurs in the same manner. The particles from the $x_i$ are blind to the discoveries of other particles. So it is possible that branching events occur multiple times at the same points. This introduces considerable dependence, but we will see that there is still independence with the branching times and offspring down a single lineage. This is enough for us to deduce that it is a.s.\ contained in a linearly expanding ball.

%

\subsection{Formal description and properties}

We need a way to make branching identically distributed while remembering the points of $\mathcal P$. To do this we will refill explored areas with points from $(\mathcal P_i)_{i \geq 1}$, a collection of independent unit Poisson point processes on $\mathbb R^d$.

%
Recall the definition of $\Gamma(x_0, \mathcal P)\subseteq \mathbb R^d$ the \emph{continuum percolation cluster of $\mathcal P$ containing $x_0$}. It is obtained recursively by setting $\Gamma(x_0,\mathcal P) = \bigcup_{i=1}^\infty \Gamma_i(x_0,\mathcal P)$ where
$\Gamma_0(x_0,\mathcal P) = \{x \in \mathcal P \colon \|x-x_0\| \leq r\}$, and for $i \geq 1$ 
	\begin{align*}
				\Gamma_{i}(x_0,\mathcal P) = \{ x \in \mathcal P\colon \|x - \Gamma_{i-1}(x_0,\mathcal P)\| \leq r\} .		
	\end{align*}
We let $\Gamma = \Gamma(0,\tilde{\mathcal P})$ be the cluster at the origin for $\tilde{\mathcal P}$ an independent unit intensity Poisson point process.  We define $\Gamma$ in terms of $\tilde{\mathcal P}$ because $\mathcal P$ has a ball at the origin removed.

  Let $\tau \in [0,\infty]$ be the time for a Brownian motion started at the origin to be within $r$ of a point in $\mathcal P$. Formally,
  $$\tau = \inf \{ t \colon \| B_t - \mathcal P \| = r \}.$$
%
 Recall that $\{ (B_t^x)_{t \geq 0} \colon x \in \mathcal P \cup \{0\}\}$ is the family of Brownian motion paths followed by frogs in the Brownian frog model. Let $W_t^x = \cup_{0\leq s \leq t} \B(B_s^x,r)$ be the Brownian sausage of radius $r$ generated by $(B_s^x)_{0 \leq s \leq t}$.
 
Now we describe the manner in which particles branch.  At time $\tau$ the particle started from the origin has revealed that $\mathcal P \cap W_{\tau} = \varnothing$ and that $\|B_\tau - x_0\| = r$ for some $x_0 \in \mathcal P$. Replace the points in $\mathcal P - W_{\tau}$ with $\mathcal P_1 \cap W_{\tau}.$ That is, form the set
 		$$\mathcal P' = (\mathcal P - W_{\tau}) \cup (\mathcal P_1 \cap W_{\tau}).$$
 		\thref{lem:surgery} in  the appendix describes this construction in more detail and establishes that the process $\mathcal P'$ (we call it $\eta$ in the appendix) is a unit intensity Poisson point process. 
	Let $\Gamma(x_0, \mathcal P')$ be the new continuum cluster at $x_0$. We place offspring at $x_{-1} := B_\tau$ and at the centers of the balls in this cluster. More precisely, at the points $$\Gamma(x_0, \mathcal P') \cap \mathcal P'=\{x_{-1}, x_0,x_1,\hdots, x_K\}$$
	where $K+1$ is the number of balls in the cluster. See Figure \ref{fig:branch}.

\begin{figure}
\centering 
\includegraphics[scale=.5]{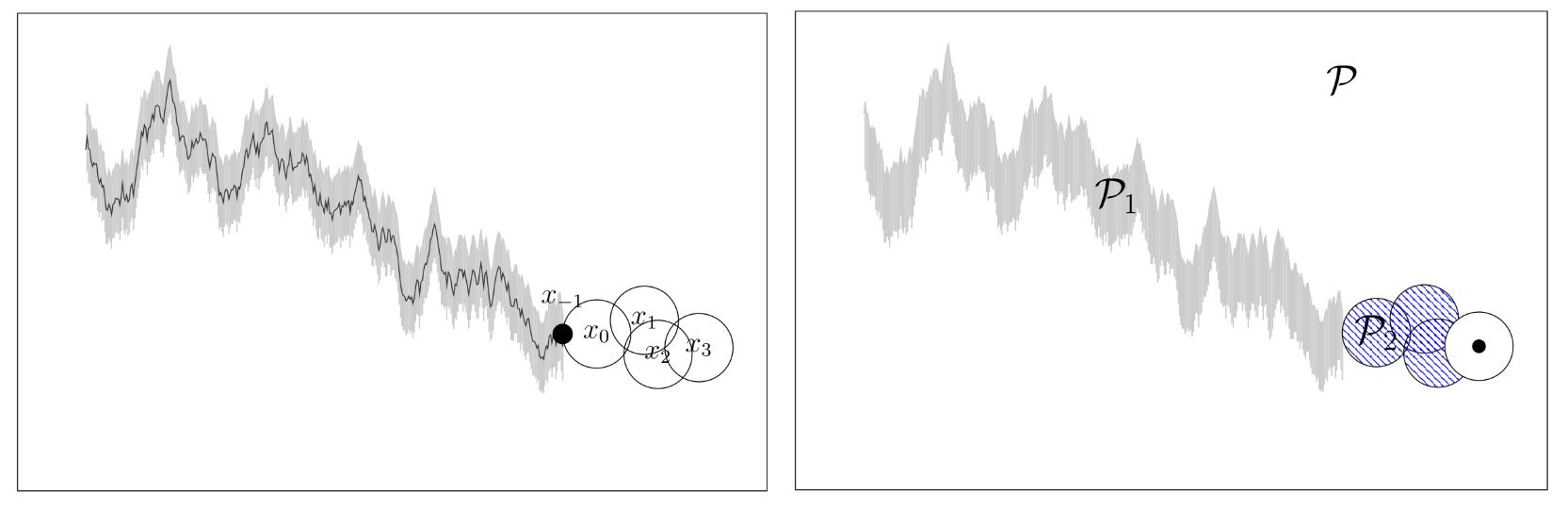}
\caption{The particle at the origin moves until it reaches a cluster. This causes a branching event that places children at $x_0, \hdots, x_K$. In this example $K=3$. The left panel shows the locations of the offspring. The right panel shows the environment that $x_3$ sees at birth. The region explored by the parent particle, $x_{-1}$, is filled by an independent Poisson point process $\mathcal{P}_1$, and the offspring cluster minus $B(x_3,r)$, shown with shaded lines, is filled with another independent Poisson process $\mathcal{P}_2$. The completely unexplored regions of the space still contain the original point process $\mathcal{P}$.}

%
\label{fig:branch}
\end{figure}

	 To ensure the environment looks the same in distribution for each new particle, we replace $\mathcal P' \cap \Gamma(x_0, \mathcal P')$ with $\mathcal P_2 \cap \Gamma(x_0, \mathcal P').$ For the particle at $x_i$ we also delete any points in $\B(x_i,r)$ from the process. This makes it so each new particle has the same environment to explore as the initial particle at the origin. Formally, each particle in $\{x_{-1},x_0,x_1, \hdots, x_K\}$ will move according to a Brownian motion in the environment 
	$$\mathcal P''_i = \Bigl[(\mathcal P' - \Gamma(x_0, \mathcal P') ) \cup (\mathcal P_2 \cap \Gamma(x_0, \mathcal P') )\Bigr] - \B(0,x_i), \qquad i=-1,0,1,\hdots,K.$$
	Notice that invariance again ensures that $\mathcal P''_i$ is a unit intensity point process on $\mathbb R^d - \B(0,x_i)$. 
	
	Now, each particle will explore $\mathcal P''_i$. The particles do not see the areas explored by their siblings. When the particle started at $x_i$ comes within $r$ of a point in $\mathcal P''_i$ it will, independent of its siblings, replace the Brownian sausage it carved out with an unused $\mathcal P_j$. We place $2+K$ offspring as before in a $\Gamma$ distributed way, then refill the cluster with $\mathcal P_{i+1}.$ 
	 All particles proceed to branch in this manner indefinitely. 
	
	Let $\zeta_t$ be the birth locations of all particles active at time $t$ in the branching process defined above. We collect a few facts regarding $\zeta_t$. The most important is that it dominates the Brownian frog model.
	
	\begin{lemma} \thlabel{lem:coupling}
 
There exists a coupling such that $\xi_t \subseteq \zeta_t$ for all $t \geq 0$. 
\end{lemma}

\begin{proof}
At all steps we use the original point process $\mathcal P$. Additionally we couple the Brownian paths in the canonical way; the frogs in both processes use the same Brownian motions, and extra frogs added into the dominating process use independent Brownian motions. When we refill explored regions of $\mathcal P$ with independent points, we are adding additional Brownian motions to the process. Since the process becomes monotonically larger when more particles are introduced, this ensures that $\xi_t  \subseteq \zeta_i$ for all $t$. 
\end{proof}

	Let $Z_n$ be the set of birthplaces for particles in the $n$th generation of the process. For a point $z_n \in Z_n$ let $z_{n-1},\hdots,z_1, z_0$ be its ancestry back to the original active particle at the origin. So, $z_{n-i}$ is the starting location of the $i$th ancestor of $z_n$. Set $\tau_i$ to be the time between the birth of $z_{i}$ and $z_{i+1}$. 
	
	\begin{lemma} \thlabel{lem:tau}
		$\tau_i \overset{d} = \tau$ and $\tau_0,\hdots, \tau_n$ are independent.
	\end{lemma}
	
	\begin{proof}

		When a particle begins a Brownian motion from $z_i$  it does so in a unit Poisson point process, $\mathcal{P}''_i$ on $\mathbb R^d-\mathbb{B}(z_i,r)$. Thus, the time to be within $r$ of a point in $\tilde { \mathcal P}$ is distributed like $\tau$. The $\tau_i$ are independent since we have no information about the unexplored areas of $\mathcal{P}''_i$, and replace the areas covered by the Brownian sausage around the path started from $z_i$ with an independent Poisson point process. 
	\end{proof}

For $1 \leq i \leq n$	let $L(z_{i-1}) \in \mathbb R^d$ be the center of the first ball reached by the particle started at $z_{i-1}$ during a branching event. Let $X_{i}$ be the number of new particles placed around $Z(z_{i-1})$ when branching occurs. 
	\begin{lemma} \thlabel{lem:K}
		It holds that $X_{i} \overset{d} = 2 + K_i$ with $K_i \overset{d} = K$ as in the definition of the branching process.  The particles counted by $X_i$ are contained in a neighborhood of $L(z_i)$ distributed like a continuum percolation cluster $\Gamma_{i}+L(z_i)$, plus a point on the boundary of the ball around $L(z_i)$. 
	\end{lemma}
	
\begin{proof}
	The reasoning is similar to the proof of \thref{lem:tau}.
\end{proof}

\subsection{Large deviations in the branching process}

We break the proof for the upper bound into three steps. In \thref{lem:avefro}, we estimate the expected number of frogs in generation $n$. In \thref{lem:maxgen}, we give an upper bound on the maximal generation present in the population at time $t$. Then we combine the second result with large deviation estimates to bound the maximal distance an active particle has traveled up to time $t$. 

\begin{lemma}\thlabel{lem:avefro}
Recall that $Z_n$ is the number of particles in the $n$th generation. Let $\mu = \E X_1$ with $X_1$ as in \thref{lem:K}. So long as $r<r_d$ it holds that $\mu < \infty$ and $EZ_n=\mu^n$. 
\end{lemma}
\begin{proof}
By \cite[Lemma 10.2]{penrose} for $r < r_d$ we have 
$\mu < \infty$. Now, \thref{lem:K} ensures that particles produce offspring in an identically distributed way. It follows that $\E[Z_{n+1}\mid Z_n] = \mu Z_n$, hence $\E Z_n = \mu^n$. Note that this computation does not require that the different families are independent.  
\end{proof}

Define the generation number of a particle in $\zeta_t$ to be the number of ancestors the particle has. Let $N_t$ be the largest generation number in $\zeta_t$.
\begin{lemma}\thlabel{lem:maxgen}
There exist $\epsilon_0,c>0$ so the $\P[ N_t > t/\epsilon_0] \leq e^{-ct}$ for sufficiently large $t$.
\end{lemma}

\begin{proof}
Consider one lineage and let $\tau_i$ be the time between the birth of two particles in the same lineage. 
By \thref{lem:tau}, the $\tau_i$ are i.i.d. Moreover, we have $\P[\tau_1 =0] = 0$ since each particle starts in an ball of radius $r$ that contains no other points in the process. 

%
%
Let $S_n=\tau_1+\cdots + \tau_n$. If $S_n<\epsilon n$, then there must be at least $n/2$ of these $\tau_i$ satisfying $\tau_i\le 2\epsilon$. 
Notice that Markov's inequality applied to \thref{lem:avefro} guarantees  $\P[Z_n \geq (3 \mu)^n] \leq 3^{-n}$. For an arbitrary $\epsilon>0$ we condition on this event to obtain
\begin{align}
\P[N_{\epsilon n} \geq n]  &\leq 3^{-n} + \P[N_{\epsilon n} \geq n \mid Z_n < (3 \mu)^n].\label{eq:bayes}
\end{align}
We can apply a union bound to the second summand to obtain
\begin{align}
\P[N_{\epsilon n} \geq n \mid Z_n < (3 \mu)^n]&\leq  (3 \mu)^n \P[S_n \leq \epsilon n] \nonumber \\
			& \le (3\mu)^n 2^n \P[\tau_1\le 2\epsilon]^{n/2},\nonumber
\end{align}
where $2^n$ is an upper bound on the number of subsets of $\{1, \ldots n \}$ of size at least $n/2$.
As $\P[\tau_1 = 0] = 0$ we can choose $\epsilon_0$ such that
\begin{equation*}
6\mu \P[\tau_1\le 2\epsilon_0]^{1/2} < 1/3.
\end{equation*}
It follows from \eqref{eq:bayes} that $\P[N_{\epsilon_0 n} \geq n] \leq  3^{-n} + 3^{-n} \leq 2^{-n}$ for large enough $n$.  
Let $t = \epsilon_0n  $ so that
$$\P[N_{t} > t/\epsilon_0] = \P[N_{\epsilon_0 n} \geq  n ]  \leq 2^{-n} = 2^{- t/\epsilon_0 } \leq e^{- c t}$$
for some $c>0$. 
\end{proof}
Next, we need a result of Penrose that the right tail of the number of balls in a continuum cluster decays exponentially.

\begin{lemma} \thlabel{lem:X}
Let $X$ be the number of offspring created at each step in the branching process as in \thref{lem:K}. It holds that $\P[X > k] \leq e^{-ck}$ for some $c>0$ and large enough $k$.
\end{lemma}
\begin{proof}
\thref{lem:K} ensures that the number of offspring is dominated by two plus the number of particles in a continuum cluster. The exponential bound then follows from the exponential bound on the cluster size in \cite[Lemma 10.2]{penrose}.
\end{proof}

Now we are in a position to prove that the frog model is contained in a linearly expanding ball.

\begin{proof}[Proof of \thref{prop:upper}]

Let $m_t = \max_{ x \in \xi_t} \|x\|$ be the furthest activated site from the origin in the frog model, and $M_t = \max_{z \in \zeta_t} \|z\|$ be the furthest birth-site of a particle in the branching process.  The coupling in \thref{lem:coupling} ensures that $m_t \preceq M_t$. So, it suffices to find a constant $R_1>0$ such that $M_t \leq R_1 t$ almost surely for all sufficiently large $t$.

Let $z \in \zeta_t$ be an arbitrary birth place of an active particle. We need to introduce a few random quantities attached to $z$ that will give us an upper bound on $\|z\|$. First we focus on the Brownian path leading to $z$. 

Suppose that the particle at $z$ was born at time $\tau \leq t$ and let $(B_s)_{s \geq 0}$ be the Brownian path followed by the ancestors of $z$, and then by the particle at $z$ after time $\tau$. Define $B^*_t = \sup_{s \leq t} \|B_s\|.$  Set $N_t$ to be the maximal generation of particles in $\zeta_t$. 
 
Now we turn our attention to the displacement of children during the branching times of ancestors of the particle started from $z$. Let $J \overset{d}= 2 r X$ with $X$, the number of particles created at a branching time, distributed as in \thref{lem:X}. Take $J_1, J_2,\hdots$ to be i.i.d.\ copies of $J$ and $(B_s)_{s \geq 0}$ to be a Brownian motion independent of the $J_i$'s.

We claim that
\begin{align}\|z\| \preceq B^*_t + \sum_{i=1}^{N_t} J_i. \label{eq:coupling_bound}
\end{align}
Note that there is dependence between $B^*_t$ and $N_t$. However, the $J_i$ are independent of both $B_t^*$ and $N_t$. 

To justify \eqref{eq:coupling_bound} we note that the contribution from the Brownian path leading to $z$ is at most $B_t^*$, and the number of jumps that have occurred is at most $N_t$. Any given offspring will displace by at most the size of a continuum cluster. This is bounded in distribution by $J=2r X$, the number of particles in the cluster times the diameter of each ball. The independence assertions regarding the $J_i$ and $B_t^*$ and $N_t$ follow from \thref{lem:K}.

Since $z \in \zeta_t$ is arbitrary it follows from \eqref{eq:coupling_bound} that 
\begin{align}M_t \preceq B^*_t + \sum_{i=1}^{N_t} J_i. \label{eq:coupling_bound2}
\end{align}
We turn our attention to providing exponential bounds on the two right-hand terms. The deviations of the maximum of Brownian motion can be bounded in a way which ensures that for any $c>0$ there exists $C>0$ such that for all sufficiently large $t$ we have,
\begin{align}
\P[ B^*_t >Ct ] \le e^{-ct} \label{eq:B*}.
\end{align}
(See eg \cite[Lemma 5.2.1]{LargeDevTA})

We can obtain a similar exponential bound on $\sum_{i=1}^{N_t} J_i$. Recall that \thref{lem:maxgen} gives $\epsilon_0$ and $c'$ so that $\P[N_t > t/\epsilon_0] \leq e^{- c't}$ for large $t$. Let $T = \lceil t / \epsilon_0 \rceil$ and $C'>0$. Using the exponential bound on the right tail of $X$ in \thref{lem:X} we have
\begin{align}
	\P\Bigl[ \sum_{i=1}^{N_t} J_i > C' t, N_t \leq T \Bigr] \leq \P\Bigl[ \sum_{i=1}^{T} J_i > C' t\Bigr]. \nonumber
\end{align}
The sum $J_1 + \cdots + J_T$ is a finite sum of i.i.d.\ random variables with exponential tails. Using the exponential Markov inequality and the exponential tails of $J_i$, one can show that for any $c>0$ there exists $C'>0$ so that 
\begin{align}\P[ J_1 + \cdots + J_T > C'T] \leq e^{- cT}.\label{eq:sum_bound}\end{align}
It follows from the relation at \eqref{eq:coupling_bound2} and the bounds in \eqref{eq:B*}, \eqref{eq:sum_bound}, and \thref{lem:maxgen} that for sufficiently large $R_1>0$ we have constants $c,c' >0$ so that
\begin{align*}\P[ M_t \geq R_1 t] &\leq \P[ B_t^* > (R_1 /2) t] + \P\Bigl[\sum_{i=1}^{N_t} J_i > (R_1/2) t\Bigr] \\
			&\leq e^{-c t} + e^{- c' t}.	
\end{align*}
This ensures that almost surely for all sufficiently large $t$ we have $M_t \leq R_1 t$ completing the proof. 
\end{proof}

\section{Passage time and containing a ball} \label{sec:linear}

First, we will show that the frog model contains a growing small cube. To do this, we will divide the $[-n,n]^d$ cube into smaller cubes using a square mesh. Then we stagger the frog's wake-up times in such a way that there is an active frog at each square in the mesh. By shrinking the size of the mesh, we are able to show that the active frogs spread out and that they will visit all frogs in $[-n,n]^d$ by time $(D+1)n^{2+\epsilon}$ for a constant $D$.
We say that events $A_n$ happen \textit{with extremely high probability (wehp)} if 
$$\P[A_n] \ge 1 - \exp(-n^{\gamma}) \hspace{.2 in} \text{for some } \gamma.$$ 
Alternatively, we say that $A_n$ happens \emph{with very high probability (wvhp)}  if $$\P[A_n] \ge 1 - n^{-m}$$ for any $m<\infty$ and sufficiently large $n$. 

We begin by stating some basic facts of Brownian motion. 
\begin{lemma}
Let $B_t$ be a one dimensional Brownian motion. If $\ell\ge 1$ is an integer 
\begin{align}
\P\left[ \max_{0 \le t \le \ell k^2} |B_t| \le k \right] \le (0.7)^\ell
\label{Bub} \\
\P\left[ \max_{0 \le t \le k^2} B_t \ge \ell k \right] \le 2 e^{-\ell^2/2}
\label{Blb}
\end{align}
\end{lemma}

\begin{proof} For the first we note that 
$$
\sup_{x \in [-k,k]} \P_x \Bigl[B_{k^2} \in [-k,k]\Bigr] \le 0.7.
$$
Here $\P_x$ is the probability when the particle starts at $x$.
For the second, the reflection principle and a standard normal tail bound gives
$$
\P\left[ \max_{0 \le t \le k^2} B_t \ge \ell k \right] \le 2\P[ B_{k^2} \ge \ell k]  \le 2 e^{-\ell^2/2}
$$
which proves the desired result.
\end{proof}  

Now we show that given $O(n^{2 +\epsilon})$ time to evolve, the Brownian frog model is very likely to wake all of the frogs in a cube of width $n$. We use the notation $f(n) = \Omega(g(n))$ to mean that there exists $C$ such that for large enough $n$ it holds $f(n) \geq C g(n)$.

\begin{lemma} \thlabel{lem:step1} Let $T_1=n^{2+\ep}$. There is a constant $D$ so that with wvhp all frogs in $[-n,n]^d$ are awake at time $(D+1)T_1$. 
\end{lemma}

\begin{proof}
Suppose $d \ge 3$. In this case all events will be wehp. Our proof will be iterative, with $D$ stages. To prepare for this, we will label our Poisson points with independent random variables uniform on $\{1, \ldots D\}$. By Poisson thinning, the frogs labeled $i$ are a Poisson process with intensity $1/D$, call the process $\mathcal{P}_i$. 

Let $\ep = 1/(3+d)$ and $\beta = 1/d$. Choose $D$ so that $(1-\beta)^{D}$ is smaller than $\frac{2+\epsilon}{2d}$. Consider $Q_n = [-n,n]^d$ with an active frog at the origin and sleeping 1-frogs at $\mathcal{P}_1$.
Run the process for $t_1 := n^{2-\ep}$. It follows from the discussion in \cite[(1.4)]{sausage1} that wehp the Brownian sausage with radius $r$ created by the active frog from the origin has volume $V \ge n^{2-2\ep}$ at time $n^{2-\ep}$. In addition, wehp the sausage is a subset of $[-n,n]^d$ because the frog is unlikely to have walked outside the box in time $n^{2-\epsilon}$. 
Since the number of points of $\mathcal{P}_1$ in a set of volume $V$ is Poisson with mean $V/D$ a standard large deviations estimate for Poisson random variables implies that wehp the number of awakened frogs is at least $ n^{2-3\ep}$. Call these the active 1-frogs.

Form a minimal covering of $Q_n$ by finitely many disjoint small cubes, $Q^1_{n^{1-\beta}},Q^2_{n^{1-\beta}},\hdots$, of side length $n_1 := n^{1-\beta}$ and call these the generation 1 (G1) cubes. Call the center of a G1 cube the cube centered at the same point with side length $n_1/2$. Continue the process until $T_1=n^{2+\ep}$. We want to bound the probability that an active 1-frog is in the center of a given G1 cube at time $T_1$. A standard hitting estimate for Brownian motion in an annulus \cite[Section 8.5.1]{durrett}  lets us write this probability as 
$$
\int_{Q^k_{n^{1-\beta}/2}} \frac{1}{(2\pi s)^{d/2}} \e^{-(x-y)^2/2s} \, dy
$$
where $s$ is the time since the 1-frog was activated and $x$ is the initial location of that 1-frog.

 Because all active 1-frogs were activated before time $t_1$, $s \in [T_1-t_1, T_1]$. Since the diameter of $Q_n$ is $2 \sqrt d n$, and all of the active 1-frogs are within $Q_n$ wehp, we have $|x - y| \leq 2 \sqrt{d} n $. So, 
\begin{align*}
\int_{Q^k_{n^{1-\beta}/2}} \frac{1}{(2\pi s)^{d/2}} \e^{-(x-y)^2/2s} \, dy &\geq (2\pi T_1)^{-d/2}|Q_{n^{1-\beta}/2}^k|  \e^{-2dn^2/(T_1-t_1)}\\
&= (2\pi n^{2+ \epsilon})^{-d/2} n^{(1-\beta)d}2^{-d} \e^{-2dn^2/(n^{2+\epsilon}-n^{2-\epsilon})} \\
&= \Omega\left(\frac{ n^{(1-\beta)d}}  { n^{(2-\ep)d/2}}\right).
\end{align*}
Thus for $n$ large, the probability that there are no 1-frogs in a given generation 1 cube at time $T_1$ is no more than
\begin{align}
\left( 1 - C \frac{ n^{(1-\beta)d}}  { n^{(2-\ep)d/2} } \right)^{n^{2-3\ep}}
& \approx \exp\left( - C n^{ (2-3\ep) + (1-\beta) d - (2-\ep) d/2 }\right) \nonumber \\
& = \exp\left( - C n^{ 1-3\ep -\ep d/2 }\right) = \exp( - Cn^\delta) \label{eq:conclusion}
\end{align}
with $\delta>0$ since $\ep = 1/(3+d)$.
Since the number of G1 cubes is less than or equal to $(2n^\beta)^d$, it follows that wehp
there is a 1-frog in the center of every G1 cube at time $T_1$.

At $T_1$, put down sleeping 2-frogs at the points in $\mathcal{P}_2$. Recall that $n_1 = n^{1-\beta}$, let $n_2 = n_1^{1-\beta}$, $t_2 = T_1 + n_1^{2-\epsilon}$, and $T_2 = T_1 + n_1^{2 + \epsilon}$. We know that wehp, there is an active 1-frog in a given G1 cube. Focusing on a specific G1 cube, we let the active 1-frog run and wake up the sleeping 2-frogs until $t_2$. 

Similar to the first step, between time $T_1$ and time $t_2$, the active 1-frog in this G1 cube will produce wehp a Brownian sausage of volume at least $n_1^{2-2\epsilon}$, waking up, wehp, at least $n_1^{2-3\epsilon}$ of the 2-frogs. Call these newly awakened frogs the active 2-frogs. Again, we allow these active 2-frogs to run from until time $T_2$. Defining generation 2 cubes as disjoint cubes of side length $n_1^{1-\beta}$ inside the G1 cube, we can do a very similar calculation to find the probability that at time $T_2$ there is an active 2-frog in the center of each G2 cube: 
$$
\int_{Q^k_{n_1^{1-\beta}/2}} \frac{1}{(2\pi s)^{d/2}} \e^{(x-y)^2/2s_2} \, dy
$$
where $s_2 \in [T_2 - t_2, T_2 - T_1]$ is the time since the 2-frog was activated and $x$ is the initial location of that 2-frog. Using the same restrictions on the maximum distance between $x$ and $y$. Note a concern might be that the initial 1-frog in each G1 cube starts near the boundary and strays out of the cube in time less than $t_2$. We can address this by requiring the more strict event that each 1-frog be inside a cube centered in $Q_{n^{1-\beta}}^k$ but with side length $1/2$ that of $Q_{n^{1-\beta}}^k$. This only changes the volume by a constant factor, and thus does not change the conclusion at \eqref{eq:conclusion}. We get that this new probability must be at least 
$$(2\pi (T_2-T_1))^{-d/2}n_1^{(1-\beta)d2^{-d}}\e^{Cn_1^2/(n_1^{2+\epsilon}-n_1^{2-\epsilon})}.$$

The above quantity is asymptotically larger than
\begin{align*}
2\pi (T_2-T_1))^{-d/2}n_1^{(1-\beta)d} &\approx \exp(-n_1^{(1-\beta)d + 2-3\epsilon - (2+\epsilon)d/2})\\
&= \exp(-n_1^{\delta})\\
&= \exp(-n^{\delta_2})
\end{align*}
for some $\delta_2 >0$. 
Thus, for large $n$ we have
$$\P[\text{there is an active $2$-frog in each $G2$ cube at time $T_2$}] = \Omega(\exp(-n^{\delta_2})).$$
So the probability that there is not an active 2-frog in a given G2 cube is once again exponentially small. Because there are $(2n_1^\beta)^d$ G2 cubes, then wehp there is an active 2-frog in the center of every G2 cube.

It is important to note here that the thinning of the frogs at the beginning was necessary to ensure that we did not reuse frogs in multiple places/steps at the same time and to provide the independent environments for the 1-frogs within the G1 cubes. 

We iterate this process with $n_k = n_{k-1}^{1-\beta}$, $t_k = T_{k-1}+n_{k-1}^{2 - \epsilon}$ and $T_k = T_{k-1} + n_{k-1}^{2+\epsilon}$ for $k$ up to $D$. We get the same results at each cube size iteration. The end result is to say that wehp we have an active $D$-frog in the center of every generation $D$ cube of side length $n^{(1-\beta)^D}$ wehp.  Because we introduced frogs in stages rather than placing them all in the process from the beginning, the frogs in this modified process wake up at the same time or later than they would in the true frog model. Therefore, if this process contains a cube of all awake frogs, we can be certain that the original process contains the same cube of all awake frogs at that time as well.    

Now we have an active frog within $n^{(1-\beta)^D}$ of every point in $[-n,n]$. For notational simplicity, define $b := d(1-\beta)^D$. All that is left is to bound the probability that a frog within $n^b$ of a point $y$ will hit $\B(y,r)$ by $T_1$. To do this, divide $[-n,n]^d$ into tiny cubes of side $r/\sqrt{d}$, so that if $y$ is in a tiny cube, $\B(y,r)$ contains the entire tiny cube. Inside each tiny cube put a tiny ball of radius $q = r/2\sqrt{d}$. Then we look at the probability that a Brownian motion starting at $x$ with $|x| = n^b$ hits $\B(0,r)$ before exiting a larger ball $\B(0,n^{2b})$. Because of our initial choice of $D$, \eqref{Blb} says that wehp the exit time of the ball $\B(0,n^{2b})$ is greater than $T_1$. Letting $S_q = \inf\{t>0\colon |B_t| = q, |B_0| = n^b\}$, then from a standard exit time probability, we get the lower bound 
$$
\P[ S_q < T_1] \ge \frac{1}{2} \frac{n^{(2-d)b} - n^{(2-d)2b}}{q^{2-d} - n^{(2-d)2b}}
$$
where the $\frac{1}{2}$ is added as a conservative estimate that the Brownian motion did not leave $\B(0,n^{2b})$.

So given a fixed tiny cube, the probability no awake frog within distance $n^b$ of it hits the tiny ball within time $T_1$ is no larger than
$$
\left( 1 -  \frac{1}{2} \frac{n^{(2-d)b} - n^{(2-d)2b}}{q^{2-d} - n^{(2-d)2b}} \right)^{n^{(d-1)b}}
\approx \exp\left( -C n^{(2-d)b + (d-1)b} \right).
$$
where the number of frogs is given by finding the number of cubes of side length $n^{(b-\beta)^D}$ within $n^b$ of a point. Since the exponent for $n$ is positive, this implies that if we divide $[-n,n]^d$ into cubes of side $r/\sqrt{d}$ then wehp each
tiny cube is visited by an awake frog by time $T_1 + T_D$, and hence all the frogs in $[-n,n]^d$ are awake by time $(D+1)T_1$.

In $d=2$, from \cite[(1.5)]{sausage1} we see that the probability that the volume of the sausage is smaller than $b t/\log t$ is only $\exp(-\gamma(b) \log t)$, where $\gamma(b)\to\infty$ as $b\to0$. So, instead of the sausage volume being at least $n^{2-2\ep}$ wehp, this only holds wvhp. Since the number of squares used in the construction is only polynomial in $n$, this wvhp result is enough to repeat the previous proof.  
\end{proof}

This result is enough to meet the conditions of Theorem \ref{thm:liggett} and to get the existence of a speed, but the statement of \thref{lem:step1} is not strong enough to say that the process contains a linearly expanding ball.  However, we will now build on this result to get the containment of a linearly expanding ball as desired. The proof goes by covering the ball of radius $n$ with smaller balls along which the process covers linearly. This happens simultaneously in all directions whp, thus covering the entire ball (See Figure \ref{fig:small_ball}).

\begin{proof}[Proof of \thref{prop:lower}]
	

 Let $L = \{t e_1 \colon t \geq 0\}$ with $e_1$ the standard unit vector. For $n \ge 0$ let $x_n$ be closest point in $\mathcal P$ to $ne_1$. 
Let $s=d/2+\alpha/2$ with $\alpha > 0$ small and let
$$
R_n = [-n^{s}, n + n^{s}] \times [-n^{s}, n^s]^{(d-1)}.
$$
By \eqref{Bub} wvhp 

\begin{enumerate}[label = (\roman*)]
 \item no frog from outside $R_n$ wakes up a frog at $x_0, x_1, \ldots x_n$ before time $n$. (to see this, note that at some time before $n$
it would have to be on the boundary of $R_n$), and 
	\item  no frog that starts in $R_n$ moves by $n^r$ before time $n$. 
\end{enumerate}

\begin{figure}
\centering
\begin{tikzpicture}[scale  =.6]
	\draw (0,0) circle (3 cm);
	\draw (0,0) -- (3,0);
	\draw (0,0) circle (.3 cm);
	\draw (.4,0) circle (.3 cm);
	\draw (.8,0) circle (.3 cm);
	\draw (1.2,0) circle (.3 cm);
	\draw (1.6,0) circle (.3 cm);
	\draw (2,0) circle (.3 cm);
	\draw (2.4,0) circle (.3 cm);
	\draw (2.8,0) circle (.3 cm);

	\draw[rotate=10] (0,0) -- (3,0);
	\draw[rotate=10] (0,0) circle (.3 cm);
	\draw[rotate=10] (.4,0) circle (.3 cm);
	\draw[rotate=10] (.8,0) circle (.3 cm);
	\draw[rotate=10] (1.2,0) circle (.3 cm);
	\draw[rotate=10] (1.6,0) circle (.3 cm);
	\draw[rotate=10] (2,0) circle (.3 cm);
	\draw[rotate=10] (2.4,0) circle (.3 cm);
	\draw[rotate=10] (2.8,0) circle (.3 cm);

\end{tikzpicture}  
\caption{We use our result from Lemma \ref{lem:step1} that $[-n,n]^d$ is covered in $n^{2+\epsilon}$ steps on a smaller scale with many overlapping balls to prove that the entire ball is covered.} \label{fig:small_ball}
\end{figure}
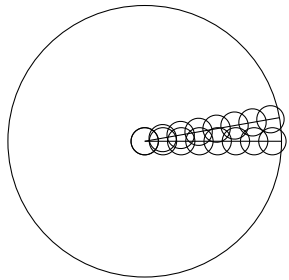

Let $q = n^{1/2 + \alpha}$. When the good event in (ii) occurs, the passage times $T(x_i,x_j)$ and $T(x_k,x_\ell)$ are independent for $k-j \ge q$ because the frogs used in finding the first passage time will not be able to travel far enough to influence the second passage time. This can be made precise by coupling independent random variables.

Let $m = \lfloor n^{\alpha} \rfloor$. Let $Y_i = T(x_{(i-1)m}, x_{im})$ for $i\le n/m$. If $m$ is large then $EY_i \le (\gamma_r + \ep)m$ by subadditivity.
Define the random variable $$\bar Y_i = Y_i \ind{ Y_i \le n^{\alpha(2+\ep)} },$$ and note that by \thref{lem:step1} wvhp $\bar Y_i = Y_i$ for all $1\le i \le n/m$.
Let $Z_i = \bar Y_i - E \bar Y_i$ and observe that $E\bar Y_i \le EY_i\le (\gamma_r + \ep)m$, and $|Z_i| \le n^{ \alpha(2+\ep) }$. 
In the definition of $Z_i$, we are subtracting off the mean of $Y_i$, but because it is bounded by something of order $n^{\alpha}$, the dominant part of $Z_i$ is still of order $n^{\alpha(2+\epsilon)}$. For $k \le n^{1/2}$ consider the sum
$$
S_k = \sum_{j=0}^{n^{1/2}/m} Z_{jn^{1/2} + k }.
$$
In words, $S_k$ represents how far the passage time deviates from the mean for a selection of points along the ray $L$. In $S_k$, we only add up passage times between points that are at least $n^{1/2 + \alpha}$ apart; so, on the good event (ii) for each fixed $k$ the variables being summed are independent, mean 0, and bounded by $n^{\alpha(2+\ep)}$. 
So, if $p$ is a positive integer
$$
\E|S_k|^{2p} = O \Bigl ( (n^{1/2}/m )^p n^{2p \alpha (2+\ep)} \Bigr).
$$
To see this, note that the $p$th power of the sum has terms of the form $Z_{i_1}^{p_1} \ldots Z_{i_k}^{p_k}$ with $p_1 + \cdots p_k = 2p$. If any $p_i=1$, then the expected value of the product vanishes. Thus, the largest contribution to the polynomial comes when all $p_i=2$. This is because $\binom{n^{1/2}/m}{p} \approx (n^{1/2}/m)^p$ is the asymptotically largest binomial coefficient. Using our bound $Z_i \leq n^{\alpha( 2 + \epsilon)}$, the dominant terms (all $p_i =2$) sum to be no larger than
$$(n^{1/2}/m)^p n^{2p \alpha( 2 + \epsilon)},$$
as claimed. The $O(\cdot)$ accounts for the smaller order terms.
Now, it follows from Markov's inequality that 
$$
\P[ |S_k| > \ep n^{1/2} ] \le \ep^{-2p} n^{-\alpha p} n^{p[2\alpha(2+\ep) - 1/2]} = \ep^{-2p} n^{-\beta p}
$$   
where $\beta = \alpha + 1/2 - 2\alpha(2+\ep)$. If $p$ is large enough then
$$
\P \left[ \sum_{k=1}^{n^{1/2} }|S_k| > \ep n \right] \le n^{-\beta p/2}.
$$
Tracing back through the definitions 
$$
\P[ T(x_0,x_n) \le n(\gamma_r + 2\ep) ] \ge 1 - n^{-\beta p/2}.
$$
If we let $N = \ell n$, then 
$$
\P [T(x_0,x_{kn}) \le kn(\gamma_r + 2\ep) \hbox{ for all $1\le k \le \ell$} ] \ge 1 - \ell n^{-\beta p/2}.
$$

The bound is the worst when $k=1$ and it is used $\ell$ times.
Given a starting point we can wvhp cover a ball of radius $(\ep N)^{1/(2+\ep)}$ wvhp in time $\ep N$. We can rewrite
$2n = (\ep\ell n)^{1/(2+\ep)}$ when $\ell = 2^{2+\ep} n^{1+\ep}/\ep$. If $p$ is large enough and we add enough 
lines then all points in the ball of radius $N$ are within $(\ep N)^{1/(2+\ep)}$ of the points on our lines 
so with high probably the ball is covered at time $N(\gamma_r + 3\ep)$ and the proof is complete. Note that to choose the lines put the ball in a cube and then pick the lines so that they are sufficiently close on the boundary of the
cube. They will be even closer inside the ball.

\end{proof}

\section{Local convergence} \label{sec:ppp}


%

\begin{proof}[Proof of Theorem \ref{thm:ppp}]
For simplicity we prove the result when $d = 2$. A similar argument with slightly harder calculus proves the result in higher dimensions.  To begin we consider a fixed square $Q = [a_1,b_1] \times [a_2,b_2]$. 
It follows from tail estimates for Brownian motion that the probability some frog awakened outside $R=[-t^{2/3},t^{2/3}]$ 
from the origin will end up in $Q$ at time $t$ tends to 0. Let ${\cal P}$ be the Poisson point process of frog positions,
and $s(x)$ be the time the frog at $x$ wakes up. The shape theorem implies 
$$
\P\left[ \sup_{x \in {\cal P} \cap R } s(x) \le  C t^{2/3} \right] \to 1.
$$
The expected number of frogs that end up in $Q$ at time $t$ is
\begin{align}
\E|A_t \cap Q| = \int_{a_1}^{b_2} dy_1  \, \int_{a_1}^{b_2} dy_2 \, \sum_{x \in {\cal P} \cap R}  p_{t-s(x)}(x,y) 
\label{absum}
\end{align}
where $p_{t-s(x)}(x,y)$ is the transition probability of Brownian motion.

Let $t_1= t -C t^{2/3}$. Let $W(y)$ denote the sum in \eqref{absum} for fixed $y$. Observe that
$$
\frac{1}{2\pi t} \sum_{x \in {\cal P}\cap Q} e^{-|y-x|^2/2t_1} \leq
W(y) \leq \frac{1}{2\pi t_1} \sum_{x \in {\cal P} \cap Q} e^{-|y-x|^2/2t}
$$
where $|z|$ is the $L^2$ norm. Scaling space by $t^{-1/2}$
$$
\frac{1}{2\pi t} \sum_{x \in t^{-1/2}({\cal P}\cap Q)} e^{-|yt^{-1/2}-x|^2/2(t_1/t)} \leq
W(yt^{-1/2}) \leq \frac{1}{2\pi t_1} \sum_{x \in t^{-1/2}({\cal P} \cap Q)} e^{-|yt^{-1/2}-x|^2/2}. 
$$

Let $a<1/2$ and let $I_{j,k} = [jn^{a},(j+1)n^{a}] \times [kn^{a},(k+1)n^{a}]$. The law of large numbers for the Poisson
process implies that $n^{-2a}|{\cal P} \cap I_{j,k}|\to 1$ in probability and in $L^2$. 
Note that the variance of $n^{-2a}|{\cal P} \cap I_{j,k}| - 1$ is $n^{-2a}$ for all $j,k$. Let $f_t(y) =\exp( - |yt^{-1/2} - z|^2/2 )$ and 
$$
g_{t,n} = \inf_{y\in n^{-1/2} I_{j,k}} f_t(y)\quad\hbox{ on $n^{-1/2} I_{j,k}$}. 
$$
A second moment calculation shows that 
$$
(2\pi t)^{-1} \sum_{z \in t^{-1/2}{\cal P} } g_{t,n}(z) \to \int dy_1 \int dy_2 \, (2\pi)^{-1} e^{-|y|^2/2} = 1
$$
in $L^2$, uniformly in $y$. From this it follows that
\begin{align*}
\E[|A_t \cap Q|] &=  \int_{a_1}^{b_2} dy_1  \, \int_{a_1}^{b_2} dy_2 \, W(y) \, dy \\
& \to  \int_{a_1}^{b_2} dy_1  \, \int_{a_1}^{b_2} dy_2 \, 1  = (b_1-a_1)(b_2-a_2) = |Q|,\\
\end{align*} 
where $|Q|$ is the area of $Q$. 

At this point we have shown expected number of frogs in $Q$ at time $t$, given the initial configuration,
converges to $|Q|$ in probability. Since the awakened frogs follow independent Brownian motions and each has a small
probability of ending up in the interval, a standard Poisson approximation result implies, see e.g., 
\cite[Theorem 3.6.6]{durrett}, that for almost every initial configuration the number of frogs in $Q$ converges in distribution to a Poisson($|Q|$) random variable. 

The extension to the joint convergence of the number of frogs in $k$ disjoint squares now follows from a general result
that does not depend on the details of our situation. 

\begin{lemma} Suppose that for each $n$, $X_{n,i}$, $i \in I_n$ are independent and take values in $\{0,1, \ldots k \}$,

(i) $\max_i P( X_{n,i} \neq 0 ) \to 0$

(ii) for each $j$ $\sum_i P( X_{n,i} = j ) \to \lambda_j$

\noindent
If $N_j = |\{ i : X_{n,i} = j \}|$ then $(N_1, \ldots N_k)$ converge to independent Poisson($\lambda_j$).
\end{lemma}

\begin{proof} It suffices to prove the result when $k=2$. It follows from the standard Poisson convergence theorem that $N_1$
converges in distribution to Poisson($\lambda_1$). Let $J_n = \{ i : X_{n,i}= j \}$. Conditional on $J_n$ the
$X_{n,i}$ with $i \not]\in J_n$ independent and take values in $\{0,2\}$ with new distribution 
$$
\hat P( X_{n,i} = 2 ) = \frac{ P(X_{n,i} = 2)}{ 1 - P( X_{n,i} = 1)}.
$$
The new variables satisfy the hypotheses of the standard Poisson convergence result so conditional on $N_1$, $N_2$ converges
to a Poisson($\lambda_2$)
\end{proof}

Consider a sequence of rectangles $vt + Q$ where $v \in \B(0,\gamma_r(1-\ep))$. Let $R = [-t^{2/3},t^{2/3}]^2$. Again it follows from
tail estimates for Brownian motion that the probability some frog awakened outside $vt+R$ will end up in $vt+Q$ at time $t$ goes to 0.
Let $s(x)$ be the time at which the frog at $x$ wakes up. Let $T_1$ be the time at which the first frog in $vt+R$ is awakened,
and let $T_2$ be the time at which the last frog in $xt+R$ is awakened. Our result about containing a small ball implies that
$P( |T_2-T_1| > t^{0.7} ) \to 1$. The shape theorem implies that $T_1/t \to |v|/\gamma_r$. Translating the previous calculation in
space and time gives the desired result.

\end{proof}

\appendix

\section{Brownian surgery of a Poisson point process}
We would like to show that we obtain a Poisson point process when we run $B_t$ until it hits a point in $\mathcal P$, delete this point, and replace the Brownian sausage carved out with a new Poisson point process. To this end, let $\{G_t\}_{t > 0}$ be a family of open sets in $\Rm^d$ such that
\[
G_0 = \{0\}, \quad \quad G_s \subset G_t, \;\; \text{for}\;\;s < t
\]
and having Lebesgue measure
\[
|G_t| = t, \quad \forall \; t > 0
\]
and such that $|\partial G_t| = 0$ for all $t$, where $\partial G_t$ is the boundary of $G_t$. Let $\mu$ be a Poisson random measure on $\Rm^d$ with uniform unit intensity. Let $\nu$ be an independent Poisson random measure on $\Rm^d$, also with uniform unit intensity. Define a random time
\[
\tau = \inf\{ t \geq 0\;|\; \mu(\overline{G_t}) \geq 1 \},
\]
which is the first time that $\overline{G_t}$ contains one of the poisson points in the $\mu$ process.

Then $\tau$ is exponentially distributed:
\[
\Pm( \tau > s ) = \Pm ( \mu(\overline{G_s}) = 0 ) = e^{-|G_s|} = e^{-s}.
\]
Because the boundary has measure zero, with probability one, we have only one point in $\overline{G_\tau}$:
\[
\mu(\overline{G_\tau}) = 1
\]
and it lies on the boundary $\partial G_\tau$.  Then define a new point process by the measure $\eta$, as follows:
\begin{align}
\eta(A) = \mu(A \cap (\overline{G_\tau})^C) + \nu(A \cap G_\tau) \label{eq:eta}
\end{align}
Thus, we use points of the $\nu$ process inside $G_\tau$, and those of the $\mu$ process in $(\overline{G_\tau})^C$. Notice, the point on the boundary is in neither set $G_\tau$ nor $(\overline{G_\tau})^C$, so it has been excluded in this construction.

\begin{lemma}\thlabel{lem:surgery}
$\eta$ is a Poisson random measure with unit intensity. 
\end{lemma}

\begin{proof}

Conditioning on $\tau$ (which is exponentially distributed), we compute
\begin{align}
\Pm( \eta(A) = k) & = \int_0^\infty \Pm( \eta(A) = k \;|\; \tau = s) e^{-s} \,ds 
\end{align}
Conditioned on $\tau = s$, $\eta(A)$ is Poisson with rate $|A|$, since it is the sum of {\em independent} Poisson random variables
\[
\mu(A \cap (\bar G_s)^C) + \nu(A \cap G_s).
\]
having rates $|A  \cap (\bar G_s)^C|$ and $|A \cap G_s|$, respectively.  Therefore,
\begin{align}
\Pm( \eta(A) = k) & = e^{-|A|} \frac{|A|^k}{k!} \int_0^\infty e^{-s} \,ds = e^{-|A|} \frac{|A|^k}{k!}.
\end{align}
So, $\eta(A)$ is Poisson with rate $|A|$. We now use characteristic functions to show independence. 

Suppose $A$ and $B$ are two disjoint measurable sets in $\Rm^d$; we now show that $\eta(A)$ and $\eta(B)$ are independent. For for any $\theta, \delta \in \Rm$, we have
\begin{align}
\expE[e^{i \theta \eta(A)}e^{i \delta \eta(B)} ] & = \int_0^\infty \expE[e^{i \theta \eta(A)}e^{i \delta \eta(B)} \;|\; \tau = s] e^{-s} \,ds \no \\
& = \int_0^\infty \expE[e^{i \theta (\mu(A \cap (\bar G_s)^C) + \nu(A \cap G_s)} e^{i \delta (\mu(B \cap (\bar G_s)^C) + \nu(B \cap G_s))} \;|\; \tau = s] e^{-s} \,ds 
\end{align}
The terms in the expectation are independent since the respective sets are all disjoint:
\begin{align}
& \expE[e^{i \theta (\mu(A \cap (\bar G_s)^C) + \nu(A \cap G_s)} e^{i \delta (\mu(B \cap (\bar G_s)^C) + \nu(B \cap G_s))} \;|\; \tau = s]  \no \\
& \quad \quad = \expE[e^{i \theta \mu(A \cap (\bar G_s)^C)}] \expE[ e^{i \theta \nu(A \cap G_s)}] \expE[ e^{i \delta \mu(B \cap (\bar G_s)^C)}] \expE[ e^{i \delta \nu(B \cap G_s))}] \no \\
& \quad \quad = e^{|A \cap (\bar G_s)^C| (e^{\theta} - 1)}  e^{|A \cap G_s| (e^{\theta} - 1)}e^{|B \cap (\bar G_s)^C| (e^{\delta} - 1)}  e^{|B \cap G_s| (e^{\delta} - 1)} \no \\
& \quad \quad = e^{|A| (e^{\theta} - 1)} e^{|B| (e^{\delta} - 1)} 
\end{align}
which is independent of $s$. The last expression is the product of the characteristic functions of two Poissons. This shows that
\[
\expE[e^{i \theta \eta(A)}e^{i \delta \eta(B)} ] =  e^{|A| (e^{\theta} - 1)} e^{|B| (e^{\delta} - 1)} = \expE[ e^{i \theta \eta(A)}] \expE[e^{i \delta \eta(B)}].
\]
This implies that $\eta(A)$ and $\eta(B)$ are independent.

\end{proof}




\bibliographystyle{alpha}
\bibliography{BM_frogs}

\subsection*{Acknowledgments} \noindent Thanks to Si Tang for her  useful suggestions
 and to James Nolen for his assistance with the argument in the appendix.


\end{document}